\documentclass{amsart}

\usepackage{lmodern}
\usepackage{microtype}
\usepackage[english]{babel}
\usepackage{hyperref}
\usepackage{xcolor} 
\usepackage{tikz} 
\usepackage{pgfplots} 
\pgfplotsset{compat=1.18}

\theoremstyle{plain} 
\newtheorem{theorem}{Theorem} 
\newtheorem{lemma}[theorem]{Lemma} 
\newtheorem{corollary}[theorem]{Corollary}

\theoremstyle{definition}
\newtheorem*{definition}{Definition}

\DeclareMathOperator{\mre}{Re}
\DeclareMathOperator{\sgn}{sgn}
\DeclareMathOperator{\spam}{span}

\begin{document} 
\title{Norm attaining vectors and Hilbert points} 
\date{\today} 

\author{Konstantinos Bampouras} 
\address{Department of Mathematical Sciences, Norwegian University of Science and Technology (NTNU), 7491 Trondheim, Norway} 
\email{konstantinos.bampouras@ntnu.no}

\author{Ole Fredrik Brevig} 
\address{Department of Mathematics, University of Oslo, 0851 Oslo, Norway} 
\email{obrevig@math.uio.no}

\begin{abstract}
	Let $H$ be a Hilbert space that can be embedded as a dense subspace of a Banach space $X$ such that the norm of the embedding is equal to $1$. We consider the following statements for a nonzero vector $\varphi$ in $H$:
	\begin{enumerate}
		\item[(A)] $\|\varphi\|_X = \|\varphi\|_H$.
		\item[(H)] $\|\varphi+f\|_X \geq \|\varphi\|_X$ for every $f$ in $H$ such that $\langle f, \varphi \rangle =0$.
	\end{enumerate}
	We use duality arguments to establish that (A) $\implies$ (H), before turning our attention to the special case when the Hilbert space in question is the Hardy space $H^2(\mathbb{T}^d)$ and the Banach space is either the Hardy space $H^1(\mathbb{T}^d)$ or the weak product space $H^2(\mathbb{T}^d) \odot H^2(\mathbb{T}^d)$. If $d=1$, then the two Banach spaces are equal and it is known that (H) $\implies$ (A). If $d\geq2$, then the Banach spaces do not coincide and a case study of the polynomials $\varphi_\alpha(z) = z_1^2 + \alpha z_1 z_2 + z_2^2$ for $\alpha\geq0$ illustrates that the statements (A) and (H) for the two Banach spaces describe four distinct sets of functions.
\end{abstract}


\maketitle

\section{Introduction}
The purpose of this paper is to introduce and study an abstract framework containing as special cases the recently investigated concepts of minimal norm Hankel operators \cite{Brevig2022} and Hilbert points \cite{BOS2023,BG2023} in addition to inner functions in Hardy spaces on polydiscs \cite{Rudin1969}. Our starting point reads as follows.

\begin{definition}
	An \emph{admissible pair} $(H,X)$ is a Hilbert space $H$ that can be embedded as a dense subspace of a Banach space $X$ such that the norm of the embedding is $1$. A nonzero vector $\varphi$ in $H$ is called \emph{norm attaining} in $X$ if $\|\varphi\|_X = \|\varphi\|_H$.
\end{definition}

Suppose that $(H,X)$ is an admissible pair and let $X^\ast$ denote the dual space of $X$. Since $H$ is a subspace of $X$ and $\|f\|_X \leq \|f\|_H$ holds for every $f$ in $H$, it is plain that every $\Psi$ in $X^\ast$ defines a bounded linear functional on $H$ and $\|\Psi\|_{H^\ast} \leq \|\Psi\|_{X^\ast}$. It follows from the Riesz representation theorem that there is $\psi$ in $H$ such that
\[\Psi(f) = \langle f, \psi \rangle\]
for every $f$ in $H$. This embeds $X^\ast$ as a subspace of $H$ and we say that a vector $\psi$ in $H$ is in $X^\ast$ when we mean that $\psi$ belongs to this subspace.

\begin{theorem} \label{thm:attain}
	Let $(H,X)$ be an admissible pair and let $\varphi$ be a nonzero vector in $H$. The following are equivalent:
	\begin{enumerate}
		\item[\normalfont (a)] $\varphi$ is norm attaining in $X$.
		\item[\normalfont (b)] $\varphi$ is in $X^\ast$ and $\|\varphi\|_{X^\ast} = \|\varphi\|_H$.
	\end{enumerate}
\end{theorem}

The conditions of Theorem~\ref{thm:attain} capture two (equivalent) ways that the Hilbert space properties of the vector in question are preserved under the embedding in $X$.

\begin{definition}
	Let $(H,X)$ be an admissible pair. A nonzero vector $\varphi$ in $H$ is called a \emph{Hilbert point} in $X$ if
	\[\|\varphi+f\|_X \geq \|\varphi\|_X\]
	holds whenever $f$ is in $H$ and $\langle f, \varphi \rangle =0$.
\end{definition}

The reasoning behind the name is that if $f$ and $\varphi$ are in $H$ and $\langle f, \varphi \rangle = 0$, then
\[\|\varphi + f\|_H = \sqrt{\|\varphi\|_H^2 + \|f\|_H^2} \geq \|\varphi\|_H,\]
by orthogonality. This definition attempts to capture that the geometry of $X$ is locally like the geometry of $H$ near the point $\varphi$. 

\begin{theorem} \label{thm:hilbertpoints}
	Let $(H,X)$ be an admissible pair and let $\varphi$ be a nonzero vector in $H$. The following are equivalent:
	\begin{enumerate}
		\item[\normalfont (c)] $\varphi$ is a Hilbert point in $X$.
		\item[\normalfont (d)] $\varphi$ is in $X^\ast$ and $\|\varphi\|_X \|\varphi\|_{X^\ast} = \|\varphi\|_H^2$.
	\end{enumerate}
\end{theorem}

Since $\|\psi\|_H^2 \leq \|\psi\|_X \|\psi\|_{X^\ast}$ plainly holds for every $\psi$ in $X^\ast$, the condition in Theorem~\ref{thm:hilbertpoints}~(d) reformulates the geometric property of a Hilbert point to a statement about a general estimate that is attained. As a consequence, we have the following.

\begin{corollary} \label{cor:trivial}
	Let $(H,X)$ be an admissible pair. If a nonzero vector $\varphi$ in $H$ is norm attaining in $X$, then $\varphi$ is a Hilbert point in $X$.
\end{corollary}

The proofs of Theorem~\ref{thm:attain} and Theorem~\ref{thm:hilbertpoints} are fairly direct consequences of the Hahn--Banach theorem and the Hilbert space structure of $H$. 

We are particularly interested in two classes of admissible pairs. To set the stage for the first class, let $\mathbb{T}$ denote the unit circle in the complex plane. The $d$-fold cartesian product $\mathbb{T}^d = \mathbb{T} \times \mathbb{T} \times \cdots \times \mathbb{T}$ becomes a compact abelian group under coordinate-wise multiplication and its Haar measure coincides with the product measure generated by the normalized Lebesgue arc length measure on $\mathbb{T}$. For $1 \leq p < \infty$, we define the Hardy space $H^p(\mathbb{T}^d)$ as the closure in $L^p(\mathbb{T}^d)$ of the set of polynomials in $d$ complex variables.

The first admissible pair of interest is $(H,X)$ with $H = H^2(\mathbb{T}^d)$ and $X = H^1(\mathbb{T}^d)$. Since a nontrivial function in $H^1(\mathbb{T}^d)$ can only vanish on a set of measure $0$ on $\mathbb{T}^d$ (see e.g.~\cite[Theorem~3.3.5]{Rudin1969}), it follows from the Cauchy--Schwarz inequality that $\varphi$ is norm attaining in $H^1(\mathbb{T}^d)$ if and only if $|\varphi|$ is constant and nonzero almost everywhere on $\mathbb{T}^d$. This is equivalent to the assertion that $\varphi = CI$ for a constant $C\neq0$ and an inner function $I$.

For this admissible pair our definition of Hilbert point is in agreement with the definition of Hilbert points in Hardy spaces from \cite{BOS2023}. Hence Corollary~\ref{cor:trivial} above supplies a simpler proof of the case $p=1$ of \cite[Corollary~2.5]{BOS2023}, which asserts that constant multiplies of inner functions are Hilbert points in $H^1(\mathbb{T}^d)$. The results in \cite{BOS2023} also demonstrate that the converse statement, i.e. that all Hilbert points in $H^1(\mathbb{T}^d)$ are constant multiples of inner functions, is true if and only if $d=1$.

In our second admissible pair of interest, $H$ is a functional Hilbert space \cite[\S36]{Halmos1982} on a nonempty set $\Omega$. We will additionally assume that the constant functions (on $\Omega$) are elements of $H$ and that the multiplier algebra $M(H)$ is dense in $H$. Moreover, we will normalize the norm of $H$ such $\|1\|_H = 1$.

The Banach space $X$ for this admissible pair will be the \emph{weak product space} $H \odot H$ which equals the collection of all functions $f$ on $\Omega$ that enjoy a \emph{weak factorization}
\begin{equation} \label{eq:weakfactorization}
	f = \sum_{j=1}^\infty g_j h_j,
\end{equation}
for sequences $(g_j)_{j\geq1}$ and $(h_j)_{j\geq1}$ in $H$ such that
\begin{equation} \label{eq:weakprodnorm}
	\sum_{j=1}^\infty \|g_j\|_H \|h_j\|_H < \infty.
\end{equation}
The norm of $H\odot H$ is the infimum of \eqref{eq:weakprodnorm} over all possible weak factorizations \eqref{eq:weakfactorization}. We refer to \cite[Theorem~2.1]{AHMR2021} for a proof that $H\odot H$ is a Banach space.

We will say that a given weak factorization \eqref{eq:weakfactorization} is \emph{optimal} should it attain this infimum. The additional assumptions on $H$ ensure that $\|f\|_{H\odot H} \leq \|f\|_H$ for every $f$ in $H$ and that $M(H)$ (and hence $H$) is dense in $H \odot H$, so $(H,H\odot H)$ is an admissible pair. It is plain that a function $\varphi$ is norm attaining in $H \odot H$ if and only if an optimal weak factorization of $\varphi$ is $\varphi = \varphi \cdot 1$.

The assumptions on $H$ also allow us to invoke \cite[Theorem~2.5]{AHMR2021}, which asserts that there is an antilinear isometric isomorphism from the dual space of $H \odot H$ to the space of all bounded Hankel operators on $H$. It follows from this and Theorem~\ref{thm:attain} that if $H = H^2(\mathbb{T}^d)$, then the requirement that an optimal weak factorization of $\varphi$ is $\varphi = \varphi \cdot 1$ coincides with the definition of minimal norm Hankel operators from \cite{Brevig2022}.

This point of view was utilized by Ortega-Cerd\`a and Seip in their counter-example to an infinite-dimensional analogue of Nehari's theorem \cite{OS2012}. Their work implies, and is qualitatively equivalent to, the fact that an optimal weak factorization of $\varphi(z) = z_1+z_2$ in the weak product space $H^2(\mathbb{T}^2) \odot H^2(\mathbb{T}^2)$ is $\varphi = \varphi \cdot 1$.

It is a direct consequence of the well-known inner-outer factorization that $H^1(\mathbb{T}) = H^2(\mathbb{T}) \odot H^2(\mathbb{T})$ as sets and with equality of norms. The inner-outer factorization is also the key ingredient in the proof of \cite[Theorem~1]{Brevig2022}, which asserts that $\|\varphi\|_{(H^2(\mathbb{T}) \odot H^2(\mathbb{T}))^\ast} = \|\varphi\|_{H^2(\mathbb{T})}$ if and only if $\varphi$ is a constant multiple of an inner function. In the present context, this can be more easily seen from Theorem~\ref{thm:attain}.

For $d\geq2$, it is an important open problem in harmonic analysis (see~\cite{HTV21}) whether there is an absolute constant $C_d>0$ such that $\|f\|_{H^1(\mathbb{T}^d)} \geq C_d \|f\|_{W(\mathbb{T}^d)}$ for every $f$ in $W(\mathbb{T}^d) = H^2(\mathbb{T}^d) \odot H^2(\mathbb{T}^d)$.

The work of Ortega-Cerd\`a and Seip discussed above shows that $C_2 \leq 2\sqrt{2}/\pi<1$. A minor improvement can be found in \cite[Theorem~5]{Brevig2022}. Since plainly
\begin{equation} \label{eq:ineqs}
	\|f\|_{H^1(\mathbb{T}^d)} \leq \|f\|_{W(\mathbb{T}^d)} \leq \|f\|_{H^2(\mathbb{T}^d)},
\end{equation}
the open problem is to ascertain whether $H^1(\mathbb{T}^d)$ and $W(\mathbb{T}^d)$ are equal as sets. Note that \eqref{eq:ineqs} also shows that if $\varphi$ is norm attaining in $H^1(\mathbb{T}^d)$, then $\varphi$ is norm attaining in $W(\mathbb{T}^d)$. This inspires us to compare the admissible pairs $(H^2(\mathbb{T}^2),H^1(\mathbb{T}^2))$ and $(H^2(\mathbb{T}^2),W(\mathbb{T}^2))$ in detail. Our case study is concerned with the polynomials
\[\varphi_\alpha(z) = z_1^2 + \alpha z_1 z_2 + z_2^2\]
for $\alpha \geq 0$. In order to state our result, we let $\alpha_0=1.62420\ldots$ denote the unique (see Lemma~\ref{lem:uniquesol}) solution of the equation
\[\sqrt{4-\alpha^2} = \frac{2}{\alpha}\arcsin{\frac{\alpha}{2}}\]
on the interval $(0,2)$.
\begin{theorem} \label{thm:casestudy}
	Suppose that $\varphi_\alpha(z) = z_1^2 + \alpha z_1 z_2 + z_2^2$ for $\alpha\geq0$. Then
	\begin{enumerate}
		\item[\normalfont (i)] $\varphi_\alpha$ is never norm attaining in $H^1(\mathbb{T}^2)$;
		\item[\normalfont (ii)] $\varphi_\alpha$ is a Hilbert point in $H^1(\mathbb{T}^2)$ if and only if $\alpha=0$ or if $\alpha=\alpha_0$;
		\item[\normalfont (iii)] $\varphi_\alpha$ is norm attaining in $W(\mathbb{T}^2)$ if and only if $0 \leq \alpha \leq 1/2$;
		\item[\normalfont (iv)] $\varphi_\alpha$ is a Hilbert point in $W(\mathbb{T}^2)$ if and only if $0 \leq \alpha \leq 1/2$ or if $\alpha=2$.
	\end{enumerate}
\end{theorem}

The main novelty of Theorem~\ref{thm:casestudy} is the assertions (ii) and (iv). The assertion (i) is trivial, since $\varphi_\alpha$ does not have constant modulus on $\mathbb{T}^2$. Taking into account Theorem~\ref{thm:attain}, we note that Theorem~\ref{thm:casestudy}~(iii) is equivalent to \cite[Theorem~10~(a)]{Brevig2022}.

As in the proof of Theorem~\ref{thm:attain} and Theorem~\ref{thm:hilbertpoints}, the main idea in our approach to Theorem~\ref{thm:casestudy} is duality. In the case that $X=H^1(\mathbb{T}^2)$ we will rely on the Riesz representation theorem for $L^1(\mathbb{T}^2)$ and in the case that $X = W(\mathbb{T}^2)$ our arguments will involve Hankel operators on $H^2(\mathbb{T}^2)$.

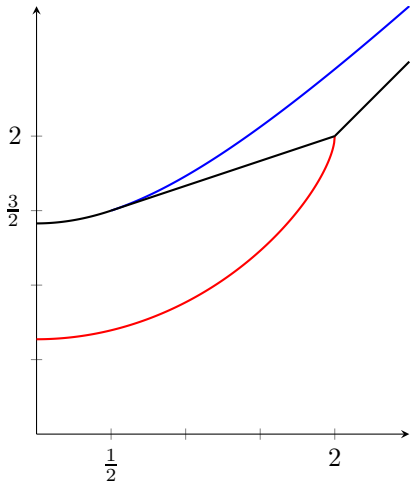
\begin{figure}
	\centering
	\begin{tikzpicture}
		\begin{axis}[
			axis equal image,
			axis lines=middle,
			ymin = 0,
			axis line style=thin,
			xtick={0.5,1,1.5,2},
			xticklabels={$\frac{1}{2}$,,,$2$},
			ytick={0.5,1,1.5,2},
			yticklabels={,,$\frac{3}{2}$,$2$},
			trig format plots=rad,
			]
			\addplot [domain=1/2:2.5, samples=500, color=blue, thick] {sqrt(2+x^2)};
			\addplot [domain=0:2, samples=500, color=red, thick] {(2*x/pi)*asin(x/2)+(2/pi)*sqrt(1-x^2/4)};
			\addplot [domain=0:1/2, samples=500, color=black, thick] {sqrt(2+x^2)};
			\addplot [domain=1/2:2, samples=2, color=black, thick] {(4+x)/3};
			\addplot [domain=2:2.5, samples=2, color=black, thick] {x};
		\end{axis}
	\end{tikzpicture}
	\caption{Norms of $\varphi_\alpha(z) = z_1^2 + \alpha z_1 z_2 + z_2^2$ for $0\leq \alpha \leq 2.5$. From top to bottom: {\color{blue} $H^2(\mathbb{T}^2)$}, $H^2(\mathbb{T}^2) \odot H^2(\mathbb{T}^2) $, and {\color{red} $H^1(\mathbb{T}^2)$}.}
	\label{fig:norms}
\end{figure}

Our efforts towards the proof of Theorem~\ref{thm:casestudy} have two remarkable byproducts. First, we can determine for which $\alpha \geq 0$ either of the general equalities in \eqref{eq:ineqs} are attained. See Figure~\ref{fig:norms}. Second, we are able to find optimal weak factorizations of $\varphi_\alpha$ for every $\alpha\geq0$. We defer the precise statements to Section~\ref{sec:casestudy} below.

Theorem~\ref{thm:casestudy} illustrates in a striking way how norm attaining vectors and Hilbert points for the two Banach spaces $H^1(\mathbb{T}^2)$ and $W(\mathbb{T}^2)$ describe four distinct classes of functions. This stands in stark contrast to the case $d=1$ where the four classes all coincide (with constant multiples of inner functions). It is clear that the inner-outer factorization has a strong impact on the situation in the latter case. 

If the functional Hilbert space $H$ is a normalized complete Pick space, then $H$ and $H \odot H$ enjoy an analogue of the inner-outer factorization (see Theorem~1.4 and Theorem~1.12 in \cite{AHMR2023}). It would be interesting to know what can be said of the norm attaining vectors and Hilbert points in this context.

\subsection*{Organization} The present paper is organized into two further sections. The next section contains the proof of Theorem~\ref{thm:attain} and Theorem~\ref{thm:hilbertpoints}. Section~\ref{sec:casestudy} is devoted to the case study of $\varphi_\alpha$ and culminates with the proof of Theorem~\ref{thm:casestudy}.

\section{Proof of Theorem~\ref{thm:attain} and Theorem~\ref{thm:hilbertpoints}}

\begin{proof}[Proof of Theorem~\ref{thm:attain}]
	We begin with the easiest implication (b) $\implies$ (a). Suppose that $\varphi$ is in $X^\ast$ and that $\|\varphi\|_{X^\ast}=\|\varphi\|_H$. Then 
	\[\|\varphi\|_H = \|\varphi\|_{X^\ast} \geq \frac{|\langle \varphi, \varphi \rangle|}{\|\varphi\|_X} = \frac{\|\varphi\|_H^2}{\|\varphi\|_X},\]
	so $\|\varphi\|_X \geq \|\varphi\|_H$ and, consequently, $\|\varphi\|_X = \|\varphi\|_H$.
	
	For the implication (a) $\implies$ (b), suppose that $\varphi$ is in $H$ and that $\|\varphi\|_X = \|\varphi\|_H$. By the Hahn--Banach theorem, there is some $\psi$ in $X^\ast$ such that $\|\psi\|_{X^\ast} = 1$ and such that $\langle \varphi, \psi \rangle = \|\varphi\|_X$. If $g$ is in $\ker{\psi}$ (i.e. if $g$ is in $X$ and $\langle g, \psi \rangle =0$), then the properties of $\psi$ ensure that $\|\varphi+g\|_X \geq |\langle \varphi + g, \psi \rangle| = \|\varphi\|_X$. This means that if $g$ is in $H \cap \ker{\psi}$, then 
	\[\|\varphi+\alpha g\|_H \geq \|\varphi + \alpha g\|_X \geq \|\varphi\|_X = \|\varphi\|_H\]
	for every complex number $\alpha$. This is equivalent to $2\mre{\left(\alpha \langle g, \varphi\rangle\right)} + |\alpha|^2 \|g\|_H^2 \geq 0$, which holds for all complex numbers $\alpha$ if and only if $\langle g, \varphi \rangle =0$. Every function $f$ in $H$ may be decomposed as
	\[f = \left(f-\frac{\langle f, \psi \rangle}{\|\varphi\|_X} \varphi\right) + \frac{\langle f, \psi \rangle}{\|\varphi\|_X} \varphi\]
	by the assumption that $\varphi$ is in $H$. The first term is in $H \cap \ker{\psi}$ since $\langle \varphi, \psi \rangle = \|\varphi\|_X$, and so it is orthogonal to $\varphi$ by the above. This means that
	\[|\langle f, \varphi \rangle| = \frac{|\langle f, \psi \rangle |}{\|\varphi\|_X} \|\varphi\|_H^2 \leq \|f\|_X \|\varphi\|_H,\]
	where we in the final estimate used that $\|\psi\|_{X^\ast}=1$ and that $\|\varphi\|_X = \|\varphi\|_H$. Since $H$ is dense in $X$, we infer from this that $\varphi$ is in $X^\ast$ and that $\|\varphi\|_{X^\ast} \leq \|\varphi\|_H$.
\end{proof}

\begin{proof}[Proof of Theorem~\ref{thm:hilbertpoints}]
	We begin with the proof that (d) $\implies$ (c). Suppose that $\varphi$ is in $X^\ast$ and that $\|\varphi\|_X \|\varphi\|_{X^\ast} = \|\varphi\|_H^2$. If $f$ is in $H$ and $\langle f, \varphi \rangle = 0$, then 
	\[\|\varphi + f\|_X \geq \frac{|\langle \varphi + f, \varphi \rangle|}{\|\varphi\|_{X^\ast}} = \frac{\|\varphi\|_H^2}{\|\varphi\|_{X^\ast}} = \|\varphi\|_X.\]
	
	For the proof that (c) $\implies$ (d), we suppose that $\varphi$ is a Hilbert point in $X$. Since $\varphi$ is in $H$ by assumption, we can decompose any $f$ in $H$ as
	\[f = \left(f-\frac{\langle f, \varphi \rangle}{\|\varphi\|_H^2} \varphi\right)+\frac{\langle f, \varphi \rangle}{\|\varphi\|_H^2} \varphi.\]
	The first term is orthogonal to $\varphi$ by construction, so the assumption that $\varphi$ is a Hilbert point in $X$ ensures that
	\[\|f\|_X \geq \frac{|\langle f, \varphi \rangle|}{\|\varphi\|_H^2} \|\varphi\|_X.\]
	Since $H$ is dense in $X$, it follows from this that $\|\varphi\|_H^2 \geq \|\varphi\|_X \|\varphi\|_{X^\ast}$.
\end{proof}

\section{A case study} \label{sec:casestudy}
A small amount of preparation is required before we can approach the proof of Theorem~\ref{thm:casestudy}. We begin by recalling that a function $f$ in $L^1(\mathbb{T}^d)$ is uniquely determined by the \emph{Fourier coefficients}
\begin{equation} \label{eq:fouriercoeff}
	\widehat{f}(\kappa) = \int_{[0,2\pi]^d} f(e^{i\theta_1},e^{i\theta_2},\ldots,e^{i\theta_d}) \, e^{-i(\kappa_1 \theta_1+\kappa_2 \theta_2 + \cdots + \kappa_d \theta_d)} \,\frac{d\theta_1}{2\pi}\frac{d\theta_2}{2\pi}\cdots\frac{d\theta_d}{2\pi},
\end{equation}
where the multi-index $\kappa=(\kappa_1,\kappa_2,\ldots,\kappa_d)$ runs over the index set $\mathbb{Z}^d$. In particular, a function $f$ in $L^1(\mathbb{T}^d)$ is in the Hardy space $H^1(\mathbb{T}^d)$ if and only if $\widehat{f}(\kappa)=0$ whenever $\kappa_j < 0$ for at least one $1 \leq j \leq d$. The set $\{z^\kappa\}_{\kappa \in \mathbb{Z}^d}$ forms an orthonormal basis for the Hilbert space $L^2(\mathbb{T}^d)$ and we will call it the \emph{standard} basis. Let also $P$ stand for the orthogonal projection from $L^2(\mathbb{T}^d)$ to $H^2(\mathbb{T}^d)$.

The following result is contained in \cite[Theorem~2.2~(a)]{BOS2023}, but we include a complete account of the  proof to illustrate its interaction with Theorem~\ref{thm:hilbertpoints}. In its statement, we will write $\sgn{z} = \frac{z}{|z|}$ if $z$ is a nonzero complex number and $\sgn{z}=0$ if $z=0$. 

\begin{lemma} \label{lem:RHP}
	A nontrivial function $\varphi$ in $H^2(\mathbb{T}^d)$ is a Hilbert point in $H^1(\mathbb{T}^d)$ if and only if 
	\begin{equation} \label{eq:RHP}
		P(\sgn{\varphi}) =  \frac{\|\varphi\|_{H^1(\mathbb{T}^d)}}{\|\varphi\|_{H^2(\mathbb{T}^d)}^2}\varphi.
	\end{equation}
\end{lemma}

\begin{proof}
	Suppose that \eqref{eq:RHP} holds. If $f$ is in $H^2(\mathbb{T}^d)$ and $\langle f, \varphi \rangle = 0$, then $\langle f, \sgn{\varphi} \rangle = 0$. Consequently,
	\[\|\varphi\|_{H^1(\mathbb{T}^d)} = \langle \varphi, \sgn{\varphi} \rangle = \langle \varphi+f, \sgn{\varphi} \rangle \leq \|\varphi+f\|_{H^1(\mathbb{T}^d)},\]
	which demonstrates that $\varphi$ is a Hilbert point in $H^1(\mathbb{T}^d)$.
	
	Suppose that $\varphi$ is a Hilbert point in $H^1(\mathbb{T}^d)$. By Theorem~\ref{thm:hilbertpoints}, we know that $\varphi$ is in the dual space of $H^1(\mathbb{T}^d)$. If we consider $H^1(\mathbb{T}^d)$ as a subspace of $L^1(\mathbb{T}^d)$, then it follows from the Hahn--Banach theorem and the Riesz representation theorem for $L^1(\mathbb{T}^d)$ that there is a function $\psi$ in $L^\infty(\mathbb{T}^d)$ such that $P\psi = \varphi$ and such that $\|\psi\|_{L^\infty(\mathbb{T}^d)} = \|\varphi\|_{(H^1(\mathbb{T}^d))^\ast}$. When combined with Theorem~\ref{thm:hilbertpoints}, this shows that
	\begin{equation} \label{eq:insertconstant}
		\|\psi\|_{L^\infty(\mathbb{T}^d)} = \|\varphi\|_{(H^1(\mathbb{T}^d))^\ast} = \frac{\langle \varphi, \varphi \rangle}{\|\varphi\|_{H^1(\mathbb{T}^d)}} = \frac{\langle \varphi, \psi \rangle}{\|\varphi\|_{H^1(\mathbb{T}^d)}}.
	\end{equation}
	Since $\varphi$ is a nontrivial function in $H^1(\mathbb{T}^d)$ by assumption, it can only vanish on a set of measure $0$ on $\mathbb{T}^d$ (see e.g.~\cite[Theorem~3.3.5]{Rudin1969}). Hence it follows from \eqref{eq:insertconstant} that $\varphi \overline{\psi} = |\varphi|>0$ almost everywhere on $\mathbb{T}^d$, and so there is a positive constant $C$ such that $\psi = C\sgn{\varphi}$ almost everywhere on $\mathbb{T}^d$. The constant is determined by \eqref{eq:insertconstant}.
\end{proof}

\begin{lemma} \label{lem:uniquesol}
	If $0\leq \alpha \leq 2$ and 
	\[\frac{\alpha}{2} \sqrt{4-\alpha^2} = \arcsin{\frac{\alpha}{2}},\]
	then $\alpha=0$ or $\alpha=1.62420\ldots$.
\end{lemma}

\begin{proof}
	It is plain that the equation holds for $\alpha=0$. If $\alpha>0$, then we rewrite the equation as
	\[\sqrt{4-\alpha^2} = \frac{2}{\alpha} \arcsin{\frac{\alpha}{2}}.\]
	The left-hand side decreases from $2$ to $0$, while the right-hand side increases (because $x \mapsto x/\sin{x}$ is increasing on $[0,\pi/2]$) from $1$ to $\pi/2$. It follows that there is a unique solution $0<\alpha<2$, which can easily be estimated.
\end{proof}

Let $m$ be an integer. A function $f$ in $L^1(\mathbb{T}^d)$ is called \emph{$m$-homogeneous} if the identity
\[f(e^{i\vartheta}z_1,e^{i\vartheta}z_2,\ldots,e^{i\vartheta}z_d) = e^{im\vartheta} f(z_1,z_2,\ldots,z_d)\]
holds for almost every $z$ on $\mathbb{T}^d$. It follows from \eqref{eq:fouriercoeff} that $f$ is $m$-homogeneous if and only if $\widehat{f}(\kappa) = 0$ whenever $\kappa_1+\kappa_2+\cdots+\kappa_d \neq m$. Consequently, the Hardy space $H^1(\mathbb{T}^d)$ only contains nontrivial $m$-homogeneous functions with $m\geq0$ and they are all polynomials. The following result corresponds to the statement (ii) in Theorem~\ref{thm:casestudy}.

\begin{theorem} \label{thm:H1casestudy}
	If $\varphi_\alpha(z) = z_1^2 + \alpha z_1 z_2 + z_2^2$ for $\alpha\geq0$, then $\varphi_\alpha$ is a Hilbert point in $H^1(\mathbb{T}^2)$ if and only if $\alpha=0$ or $\alpha = 1.62420\ldots$.
\end{theorem}

\begin{proof}
	We will use Lemma~\ref{lem:RHP}. Since $\varphi_\alpha$ is $2$-homogeneous, it is plain that $\sgn{\varphi_\alpha}$ is also $2$-homogeneous. Consequently, it follows that
	\[P(\sgn{\varphi_\alpha}) = az_1^2 + b z_1 z_2 + c z_2^2.\]
	Since $\varphi_\alpha(z_2,z_1) = \varphi_\alpha(z_1,z_2)$, we must have $a=c$. Hence Lemma~\ref{lem:RHP} implies that $\varphi_\alpha$ is a Hilbert point in $H^1(\mathbb{T}^2)$ if and only if
	\begin{equation} \label{eq:solveme}
		\alpha \,\widehat{\sgn{\varphi_\alpha}}(0,2) = \widehat{\sgn{\varphi_\alpha}}(1,1).
	\end{equation}
	We begin with the latter Fourier coefficient, which is slightly simpler to compute. Here we have
	\[\left(\sgn{\varphi_\alpha}(e^{i\theta_1},e^{i\theta_2})\right)\, e^{-i(\theta_1+\theta_2)} = \sgn{\left(\alpha+2\cos(\theta_1-\theta_2)\right)},\]
	which means that
	\[\widehat{\sgn{\varphi_\alpha}}(1,1) = \int_0^{2\pi} \sgn{\left(\alpha+2\cos{\vartheta}\right)}\,\frac{d\vartheta}{2\pi} = \begin{cases}
		\frac{2}{\pi} \arcsin{\frac{\alpha}{2}}, & \text{if } 0 \leq \alpha \leq 2; \\
		1, & \text{if } \alpha>2.
	\end{cases}\]
	For the former Fourier coefficient, we have
	\[\left(\sgn{\varphi_\alpha}(e^{i\theta_1},e^{i\theta_2})\right)\, e^{-2i\theta_2} = e^{i(\theta_1-\theta_2)}\sgn{\left(\alpha+2\cos(\theta_1-\theta_2)\right)}\]
	which yields
	\[\widehat{\sgn{\varphi_\alpha}}(0,2) = \int_0^{2\pi} e^{i\vartheta}\,\sgn{\left(\alpha+2\cos{\vartheta}\right)}\,\frac{d\vartheta}{2\pi} = \begin{cases}
		\frac{1}{2\pi} \sqrt{4-\alpha^2}, & \text{if } 0 \leq \alpha \leq 2; \\
		0, & \text{if } \alpha>2.
	\end{cases}\]
	We insert these formulas into the equation \eqref{eq:solveme}. There are plainly no solutions if $\alpha>2$. If $0\leq \alpha \leq 2$, then we get precisely the equation considered in Lemma~\ref{lem:uniquesol}.
\end{proof}

Before we proceed to second part of our case study, let us compute
\[\|\varphi_\alpha\|_{H^1(\mathbb{T}^2)} = \int_0^{2\pi} |\alpha+2\cos{\vartheta}|\,\frac{d\vartheta}{2\pi} = \begin{cases}
	\frac{2}{\pi}\Big(\alpha \arcsin{\frac{\alpha}{2}} + \sqrt{4-\alpha^2}\,\Big), & \text{if } 0 \leq \alpha \leq 2; \\
	\alpha, & \text{if } \alpha > 2.
\end{cases}\]
This computation and Theorem~\ref{thm:weaknorm} below forms the basis for Figure~\ref{fig:norms}.

Let $m$ be an integer and let $P_m$ denote the orthogonal projection from $L^2(\mathbb{T}^d)$ to its subspace of $m$-homogeneous functions. By orthogonality, every $f$ in $H^2(\mathbb{T}^d)$ satisfies the equation
\begin{equation} \label{eq:odecomp}
	\|f\|_{H^2(\mathbb{T}^d)}^2 = \sum_{m=0}^\infty \|P_m f\|_{H^2(\mathbb{T}^d)}^2.
\end{equation}
It is clear that $P_m$ is densely defined on the weak product space $W(\mathbb{T}^d)$. We next show that it extends to a norm $1$ operator on $W(\mathbb{T}^d)$ and, consequently, on its dual space. This result (in a slightly different context) can be found in \cite[Theorem~5]{BP2016}. In order to make the present paper self-contained, we repeat the proof.

\begin{lemma} \label{lem:mhomproj}
	If $m$ is nonnegative integer, then $P_m$ extends to a norm $1$ operator on $W(\mathbb{T}^d)$ and on $(W(\mathbb{T}^d))^\ast$.
\end{lemma}

\begin{proof}
	The first assertion implies the other by duality since $P_m$ is self-adjoint in the pairing of $H^2(\mathbb{T}^d)$. The function $f(z)=z_1^m$ shows that $\|P_m\|_{W(\mathbb{T}^d) \to W(\mathbb{T}^d)} \geq 1$. Let $f$ be a function in $W(\mathbb{T}^d)$ and let $f = \sum_{j \geq 1} g_j h_j$
	be a weak factorization of $f$. Then
	\[P_m f = \sum_{j=1}^\infty \sum_{n=0}^m P_n g_j P_{m-n} h_j,\]
	and, consequently,
	\[\|P_m f\|_{W(\mathbb{T}^d)} \leq \sum_{j=1}^\infty \sum_{n=0}^m \|P_n g_j\|_{H^2(\mathbb{T}^d)} \|P_{m-n} h_j\|_{H^2(\mathbb{T}^d)} \leq \sum_{j=1}^\infty \|g_j\|_{H^2(\mathbb{T}^d)} \|h_j\|_{H^2(\mathbb{T}^d)}\]
	where we used the Cauchy--Schwarz inequality in the inner sum and \eqref{eq:odecomp} twice.
\end{proof}

\begin{lemma} \label{lem:mhomdualnorm}
	Let $m$ be a nonnegative integer. If $\varphi$ is a nontrivial $m$-homogeneous polynomial, then there is a $m$-homogeneous polynomial $\psi$ such that
	\begin{equation} \label{eq:mhomdualnorm}
		\|\varphi\|_{W(\mathbb{T}^d)} = \frac{\langle \psi, \varphi \rangle}{\|\psi\|_{(W(\mathbb{T}^d))^\ast}}.
	\end{equation}
\end{lemma}

\begin{proof}
	Since $\varphi$ is nontrivial, it follows from the Hahn--Banach theorem and the fact that $(W(\mathbb{T}^d))^\ast$ is embedded in $H^2(\mathbb{T}^d)$ that there is $\psi$ in $H^2(\mathbb{T}^d)$ such that \eqref{eq:mhomdualnorm} holds. Since $P_m$ is self-adjoint in the pairing of $H^2(\mathbb{T}^d)$ and since $P_m \varphi = \varphi$, it follows from Lemma~\ref{lem:mhomproj} that \eqref{eq:mhomdualnorm} also holds if $\psi$ is replaced by $P_m \psi$.
\end{proof}

Let $\overline{H^2}(\mathbb{T}^d)$ be the closed subspace of $L^2(\mathbb{T}^d)$ consisting of the complex conjugates of functions in $H^2(\mathbb{T}^d)$ and let $\overline{P}$ denote the orthogonal projection from $L^2(\mathbb{T}^d)$ to $\overline{H^2}(\mathbb{T}^d)$. Suppose that $\psi$ be a function $H^2(\mathbb{T}^d)$. The formula
\[\mathbf{H}_\psi f = \overline{P}(\overline{\psi} f)\]
densely defines a Hankel operator $\mathbf{H}_\psi$ from $H^2(\mathbb{T}^d)$ to $\overline{H^2}(\mathbb{T}^d)$. In the present context, \cite[Theorem~2.5]{AHMR2021} asserts that $\mathbf{H}_\psi$ extends to a bounded linear operator if and only if $\psi$ is in $(W(\mathbb{T}^d))^\ast$ and that in this case $\|\mathbf{H}_\psi\| = \|\psi\|_{(W(\mathbb{T}^d))^\ast}$. If $\psi$ is in $(W(\mathbb{T}^d))^\ast$ and $f,g$ are in $H^2(\mathbb{T}^d)$, then 
\[\langle \mathbf{H}_\psi f, \overline{g} \rangle = \langle fg, \psi \rangle.\]
This formula makes it easy to compute the matrix of $\mathbf{H}_\psi$ with respect to the standard basis that $H^2(\mathbb{T}^d)$ and $\overline{H^2}(\mathbb{T}^d)$ inherit from $L^2(\mathbb{T}^d)$.
 
\begin{lemma} \label{lem:hankelnorm}
	If $\varphi_\alpha(z) = z_1^2 + \alpha z_1 z_2 + z_2^2$ for $\alpha\geq0$, then
	\[\|\varphi_\alpha\|_{(W(\mathbb{T}^2))^\ast} = \max\big(\sqrt{2+\alpha^2},1+\alpha\big).\] 
\end{lemma}

\begin{proof}
	The matrix of the Hankel operator $\mathbf{H}_{\varphi_\alpha}$ with respect to the standard basis of $H^2(\mathbb{T}^2)$ and $\overline{H^2}(\mathbb{T}^2)$, with rows and columns containing all zeros omitted, is
	\[\begin{pmatrix}
		0 & 0 & 0 & 1 & \alpha & 1 \\
		0 & 1 & \alpha & 0 & 0 & 0 \\
		0 & \alpha & 1 & 0 & 0 & 0 \\
		1 & 0 & 0 & 0 & 0 & 0 \\
		\alpha & 0 & 0 & 0 & 0 & 0 \\
		1 & 0 & 0 & 0 & 0 & 0 \\
	\end{pmatrix}.\]
	Let $(\mathbf{e}_j)_{j=1}^6$ be the standard basis of $\mathbb{C}^6$. Due to orthogonality and the block structure of the matrix, it is sufficient to let it act on the subspaces $\spam\{\mathbf{e}_1\}$, $\spam\{\mathbf{e}_2,\mathbf{e}_3\}$, and $\spam\{\mathbf{e}_4,\mathbf{e}_5,\mathbf{e}_6\}$. The norms are, respectively, $\sqrt{2+\alpha^2}$, $1+\alpha$, and $\sqrt{2+\alpha^2}$.
\end{proof}

We mention in passing that the block structure of the matrix appearing in the proof of Lemma~\ref{lem:hankelnorm} is a special case of a general phenomenon that occurs for Hankel operators on $H^2(\mathbb{T}^d)$ with $m$-homogeneous symbols (see \cite[Theorem~4]{Brevig2022}).

Lemma~\ref{lem:hankelnorm} allows us to compute one of the two nontrivial quantities in the condition of Theorem~\ref{thm:hilbertpoints}~(d) for the polynomials $\varphi_\alpha$. It is also the crucial ingredient in the following result.

\begin{theorem} \label{thm:weaknorm}
	Suppose that $\varphi_\alpha(z) = z_1^2 + \alpha z_1 z_2 + z_2^2$. Then
	\[\|\varphi_\alpha\|_{W(\mathbb{T}^2)} = \begin{cases}
		\sqrt{2+\alpha^2}, & \text{if } 0 \leq \alpha \leq 1/2; \\
		\frac{4+\alpha}{3}, & \text{if } 1/2 < \alpha \leq 2; \\
		\alpha, & \text{if } \alpha>2.
	\end{cases}\]
\end{theorem}

\begin{proof}
	By Lemma~\ref{lem:mhomdualnorm} there is a $2$-homogeneous polynomial $\psi(z) = az_1^2 + b z_1 z_2 + c z_2^2$ such that
	\[\|\varphi_\alpha\|_{W(\mathbb{T}^2)} = \frac{\langle \psi, \varphi_\alpha \rangle}{\|\psi\|_{(W(\mathbb{T}^2))^\ast}}.\]
	It follows from triangle inequality (for $(W(\mathbb{T}^2))^\ast$) that if this formula holds for $\psi_1$ and $\psi_2$, then it also holds for $\psi_1+\psi_2$. Since the coefficients of $\varphi_\alpha$ are real, it follows that $a$, $b$, and $c$ are real. Moreover, since $\varphi_\alpha(z_2,z_1)=\varphi_\alpha(z_1,z_2)$ we must have $a=c$. We consider first the case that $a=c\neq0$, where we normalize $\psi$ with $a=c=1$ and $b=\beta\geq0$. Using Lemma~\ref{lem:hankelnorm} we get that
	\[\|\varphi_\alpha\|_{W(\mathbb{T}^2)} = \sup_{\beta\geq 0} F_\alpha(\beta) \qquad \text{for} \qquad F_\alpha(\beta) = \begin{cases}
		\frac{2+\alpha \beta}{\sqrt{2+\beta^2}}, & \text{if } 0 \leq \beta \leq 1/2; \\
		\frac{2+\alpha \beta}{1+\beta}, & \text{if } \beta>1/2.
	\end{cases}\]
	There are three cases to consider:
	\begin{enumerate}
		\item[\normalfont (i)] If $0 \leq \alpha \leq 1/2$, then $F_\alpha$ is increasing to $\beta=\alpha$ and then decreasing.
		\item[\normalfont (ii)] If $1/2 < \alpha \leq 2$, then $F_\alpha$ is increasing to $\beta=1/2$ and then decreasing.
		\item[\normalfont (iii)] If $\alpha>2$, then $F_\alpha$ is increasing.
	\end{enumerate}
	Note that to attain (iii) we have to let $\beta \to \infty$. This is equivalent to the case $a=c=0$ that we excluded above. The proof is completed by computing 
	\[F_\alpha(\alpha) = \sqrt{2+\alpha^2},\qquad F_\alpha(1/2) = \frac{4+\alpha}{3},\qquad  \text{and that}\qquad F_\alpha(\beta) \to \alpha\]
	as $\beta \to \infty$.
\end{proof}

The knowledge of $\|\varphi_\alpha\|_W$ from Theorem~\ref{thm:weaknorm} makes it possible to guess an optimal weak factorization \eqref{eq:weakfactorization} of $\varphi_\alpha$ in the three cases.

\begin{enumerate}
	\item[\normalfont (i)] If $0 \leq \alpha \leq 1/2$, then an optimal weak factorization is $\varphi_\alpha = \varphi_\alpha \cdot 1$.
	\item[\normalfont (ii)] If $1/2<\alpha<2$, then an optimal weak factorization is
	\[\varphi_\alpha(z) = \frac{2}{3}(\alpha-1/2) (z_1+z_2)(z_1+z_2) + \frac{2}{3}(2-\alpha) \left(z_1^2+\frac{z_1z_2}{2}+z_2^2\right) \cdot 1.\]
	\item[\normalfont (iii)] If $\alpha>2$, then an optimal weak factorization is
	\[\varphi_\alpha(z) = \left(z_1 + \frac{\alpha+\sqrt{\alpha^2-4}}{2} z_2\right)\left(z_1 + \frac{\alpha-\sqrt{\alpha^2-4}}{2} z_2\right).\]
\end{enumerate}

We conclude the paper by wrapping up the proof of Theorem~\ref{thm:casestudy}.	

\begin{proof}[Proof of Theorem~\ref{thm:casestudy}]
	The statement (i) is trivial since $\varphi_\alpha$ has constant modulus on $\mathbb{T}^2$ for no $\alpha$ and---as noted above---the statement (ii) is the same as Theorem~\ref{thm:H1casestudy}. It is plain that $\|\varphi_\alpha\|_{H^2(\mathbb{T}^2)} = \sqrt{2+\alpha^2}$. To settle (iii) and (iv) we use, respectively, Theorem~\ref{thm:attain} and Theorem~\ref{thm:hilbertpoints} that require us to solve the equations
	\[\|\varphi_\alpha\|_{(W(\mathbb{T}^2))^\ast} = \sqrt{2+\alpha^2} \qquad \text{and} \qquad \|\varphi_\alpha\|_{W(\mathbb{T}^2)} \|\varphi_\alpha\|_{(W(\mathbb{T}^2))^\ast} = 2+\alpha^2.\]
	We use Lemma~\ref{lem:hankelnorm} to see that the first equation holds if and only if $0\leq \alpha \leq 1/2$. We then use both Lemma~\ref{lem:hankelnorm} and Theorem~\ref{thm:weaknorm} to see that the second equation holds if and only if $0\leq \alpha \leq 1/2$ or $\alpha=2$.
\end{proof}

\bibliographystyle{amsplain} 
\bibliography{nahp}

\end{document}